\documentclass[11pt,fleqn]{article}

\usepackage{amsmath,amssymb,amsthm,enumerate,cite}

\newcommand{\p}{\partial}

\newcommand{\const}{\mathop{\rm const}\nolimits}

\newtheorem{theorem}{Theorem}

{\theoremstyle{definition}
\newtheorem{definition}{Definition}
\newtheorem{note}{Note}
\newtheorem*{note*}{Note}

}

\begin{document}

\par\noindent
{\LARGE\bf On nonclassical symmetries of generalized Huxley equations\par}

{\vspace{4mm}\par\noindent {\bf N.M. Ivanova$^{\dag,\ddag}$ and C. Sophocleous$^\ddag$
} \par\vspace{2mm}\par}

{\vspace{2mm}\par\noindent {\it
$^{\dag}$~Institute of Mathematics of NAS of Ukraine, 
3 Tereshchenkivska Str., 01601 Kyiv, Ukraine\\
}}
{\noindent \vspace{2mm}{\it $\phantom{^{\dag,\ddag}}$~e-mail: ivanova@imath.kiev.ua }\par}

{\vspace{2mm}\par\noindent {\it
$^{\ddag}$~Department of Mathematics and Statistics, University of Cyprus,
CY 1678 Nicosia, Cyprus\\
}}
{\noindent {\it
$\phantom{^\ddag}$~e-mail: christod@ucy.ac.cy
} \par}

{\vspace{5mm}\par\noindent\hspace*{8mm}\parbox{140mm}{\small
Nonclassical symmetries of  a class of generalized Huxley equations of form $u_t=u_{xx}+k(x)u^2(1-u)$ are found.
More precisely, for the class under consideration we completely classify reduction operators with $\tau=1$ and give
a wide number of examples of equations admitting reduction operators with $\tau=0$.
}\par\vspace{5mm}}

\section{Introduction}

Consider reaction-diffusion equation of form
\begin{equation}\label{eqReacDifEqCubicSpace}
u_t=u_{xx}+k(x)u^2(1-u),
\end{equation}
where $k(x)\ne0$. 
This equation models many phenomena that occur in different areas of mathematical physics and biology.
In particular, it can be used to describe the spread of a recessive advantageous allele through a population in which
there are only two possible alleles at the locus in question.
Equation~\eqref{eqReacDifEqCubicSpace} is interesting also in the area of nerve axon potentials~\cite{Scott1975}.
Case $k=\const$ is the famous Huxley equation.
For more details about application
see~\cite{Bradshaw-Hajek2004,Bradshaw-Hajek&Edwards&Broadbridge&Williams2007}
and references therein.

Bluman and Cole \cite{Bluman&Cole1969} introduced a new method for finding group-invariant
(called also similarity) solutions of
partial differential equations. The method was called by the authors ``non-classical''
to emphasize the difference between it and the ``classical'' Lie reduction method described, e.g., in~\cite{Olver1986,Ovsiannikov1982}.
A precise and rigorous definition of nonclassical invariance
was firstly formulated in~\cite{Fushchych&Tsyfra1987} as ``a generalization of the
Lie definition of invariance'' (see also~\cite{Zhdanov&Tsyfra&Popovych1999}).
Later operators satisfying the nonclassical invariance criterion were also called, by different authors, nonclassical
symmetries, conditional symmetries, $Q$-conditional symmetries and reduction operators~\cite{FushchichSerov1988,
Fushchych&Shtelen&Serov&Popovych1992,Levi&Winternitz1989}.
The necessary definitions, including ones of equivalence of reduction operators,
and relevant statements on this subject are collected in~\cite{Popovych&Vaneeva&Ivanova2005,VPS_2009}.

Bradshaw-Hajek
{\it at al}~\cite{Bradshaw-Hajek2004,Bradshaw-Hajek&Edwards&Broadbridge&Williams2007} started
studying class~\eqref{eqReacDifEqCubicSpace} from the symmetry point of view.
More precisely, they found some cases of equations~\eqref{eqReacDifEqCubicSpace} admitting Lie and/or nonclassical symmetries.
Complete classification of Lie symmetries of class~\eqref{eqReacDifEqCubicSpace}
is performed in~\cite{Ivanova2008}.
Conditional symmetries of Huxley and Burgers--Haxley equations having nontrivial intersection with class~\eqref{eqReacDifEqCubicSpace}
are investigated in~\cite{Hydon2000,Estevez1994,Estevez&Gordoa1995,ArrigoHillBroadbridge1993,Clarkson&Mansfield1993}.
The present paper is a step towards to the complete classification of nonclassical symmetries of class~\eqref{eqReacDifEqCubicSpace}.

Reduction operators
of equations~\eqref{eqReacDifEqCubicSpace} have the general form
$
Q = \tau\p_t + \xi\p_x + \eta\p_u,
$
where~$\tau$, $\xi$ and~$\eta$ are functions of $t$, $x$ and $u$, and $(\tau, \xi) \ne (0,0)$. We consider two cases:

\noindent
1. $\tau \ne 0$. Without loss of generality $\tau=1$.

\smallskip
\noindent
2. $\tau=0,~\xi\ne 0$. Without loss of generality $\xi=1$.

\smallskip
\noindent
We present a complete classification for the case 1, while for the case 2 we found several examples.

\section{Equivalence transformations and Lie symmetries}

Since classification of nonclassical symmetries is impossible without detailed knowledge of Lie invariance properties,
we review~\cite{Ivanova2008} the equivalence group and results of the group classification of class~\eqref{eqReacDifEqCubicSpace}.
The complete equivalence group~$G^{\sim}$ of class~\eqref{eqReacDifEqCubicSpace} contains
only scaling and translation transformations of independent variables~$t$ and~$x$.
More precisely it consists of transformations
\[
\tilde t=\varepsilon_1^2t+\varepsilon_2,\quad
\tilde x=\varepsilon_1x+\varepsilon_3,\quad
\tilde u=u,\quad
\tilde k=\varepsilon_1^{-2}k,
\]
where $\varepsilon_i$, $i=1,2,3$ are arbitrary constants, $\varepsilon_1\ne0$.


\begin{theorem}
There exists three $G^{\sim}$-inequivalent cases of equations from class~\eqref{eqReacDifEqCubicSpace}
admitting nontrivial Lie invariance algebras
(the values of $k$ are given together with the corresponding maximal Lie invariance algebras, $c=\const$) (see~\cite{Ivanova2008}):
\begin{gather*}
1:\quad \forall k, \quad \langle \p_t \rangle;\\
2:\quad k=c, \quad \langle \p_t,\, \p_x \rangle;\\
3:\quad k=cx^{-2}, \quad \langle \p_t,\, 2t\p_t+x\p_x \rangle.
\end{gather*}
\end{theorem}

In the following section we search for nonclassical symmetries which are not equivalent to the above Lie symmetries.

\section{Nonclassical symmetries}

Now let us recall the definition of nonclassical symmetry (or conditional symmetry, or reduction operator).

Reduction operators (nonclassical symmetries, $Q$-conditional symmetries)
of a differential equation~$\mathcal{L}$ of form $L(t,x,u_{(r)})=0$ have the general form
\[
Q = \tau\p_t + \xi\p_x + \eta\p_u,
\]
where~$\tau$, $\xi$ and~$\eta$ are functions of $t$, $x$ and $u$, and $(\tau, \xi) \ne (0,0)$.
Here $u_{(r)}$ denotes the set of all the derivatives of the function $u$ with respect to $t$ and~$x$
of order not greater than~$r$, including $u$ as the derivative of order zero.

The first-order differential function~$Q[u]:=\eta(t,x,u)-\tau(t,x,u)u_t-\xi(t,x,u)u_x$
is called the {\it characteristic} of the operator~$Q$.
The characteristic PDE $Q[u]=0$ is called also the \emph{invariant surface condition}.
Denote the manifold defined by the set of all the differential
consequences of the characteristic equation~$Q[u]=0$ in the jet space $J^{(r)}$
by ${\cal Q}^{(r)}$.

\begin{definition}\label{DefinitionOfCondSym}
The differential equation~$\mathcal{L}$ of form $L(t,x,u_{(r)})=0$ is called
\emph{conditionally (nonclassicaly) invariant} with respect to the operator $Q$ if
the relation
$Q_{(r)}L(t,x,u_{(r)})\bigl|_{\mathcal{L}\cap\mathcal{Q}^{(r)}}=0$
holds, which is called the \emph{conditional invariance
criterion}. Then $Q$ is called an operator of \emph{conditional
symmetry} (or $Q$-conditional symmetry, nonclassical symmetry, reduction operator etc)
of the equation~$\mathcal{L}$.
\end{definition}

In Definition~\ref{DefinitionOfCondSym} the symbol $Q_{(r)}$ stands for the standard $r$-th prolongation
of the operator~$Q$ \cite{Olver1986,Ovsiannikov1982}.


The classical (Lie) symmetries are, in fact, partial cases of nonclassical symmetries.
Therefore, below we solve the problem on finding only pure nonclassical symmetries which are not equivalent to classical ones.
Moreover, our approach is based on application of the notion of equivalence of nonclassical symmetries
with respect to a transformation group (see, e.g.,~\cite{Popovych&Vaneeva&Ivanova2005}).
For more details, necessary definitions and properties of nonclassical symmetries we refer the reader
to~\cite{Zhdanov&Tsyfra&Popovych1999,Popovych&Vaneeva&Ivanova2005,Popovych1998,VPS_2009}.

Since~\eqref{eqReacDifEqCubicSpace} is an evolution equation, there exist two principally different cases of finding $Q$: 1.~$\tau\ne0$ and 2. $\tau=0$.

First we consider the case with $\tau\ne0$.
Here without loss of generality we can assume that $\tau=1$. The results are summarized in the following theorem.


\begin{theorem}
All possible cases of equations~\eqref{eqReacDifEqCubicSpace} admitting nonclassical symmetries with $\tau=1$ are exhausted by the following ones:
\begin{enumerate}
\item $k=c\tan^2 x$:\quad $Q=\p_t-\cot x\p_x$,
\item $k=c\tanh^2 x$:\quad $Q=\p_t-\coth x\p_x$,
\item $k=c\coth^2 x$:\quad $Q=\p_t-\tanh x\p_x$,
\item $k=cx^{2}$:\quad $Q=\p_t-\frac1x\p_x$,
\item $k=\frac{c^2}2$ $(c>0)$:\quad $Q=\p_t\pm\frac c2(3u-1)\p_x-\frac{3c^2}{4}u^2(u-1)\p_u$
\item $k=2x^{-2}$:\quad $Q=\p_t+\frac3x(u-1)\p_x-\frac3{x^2}u(u-1)^2\p_u$,
\end{enumerate}
where $c$ is an arbitrary constant.
\end{theorem}

\begin{proof}
We search for reduction operator (operator of nonclassical ($Q$-conditional) symmetry) in form
$
Q=\p_t+\xi(t,x,u)\p_x+\eta(t,x,u)\p_u.
$
Then the system of determining equations for the coefficients of operator~$Q$ has the form
\begin{gather*}
\xi_{uu}=0,\\
2\xi\xi_u-2\xi_{xu}+\eta_{uu}=0,\\
2\xi\xi_x-2\eta\xi_u-3k\xi_uu^3+3k\xi_uu^2+2\eta_{xu}-\xi_{xx}+\xi_t=0,\\
-k\eta_uu^2(1-u)+2k\xi_xu^2(1-u)+\eta_{xx}-2\xi_x\eta\\=\eta_t-k_x\xi u^2(1-u)-2k\eta u+3k\eta u^2.
\end{gather*}
From the first equation we obtain immediately that
\[
\xi=\phi(t,x)u+\psi(t,x).
\]
Substituting it to the second equation we derive
\[
\eta=-\frac13\phi^2u^3-\phi\psi u^2+\phi_xu^2+A(t,x)u+B(t,x).
\]
Then, splitting the rest of determining equations with respect to different powers of $u$ implies the following
system of equations for coefficients $\phi$, $\psi$, $A$ and $B$.
\begin{gather}\nonumber
\frac23\phi^3-3k\phi=0,\quad
-4\phi\phi_x+2\phi^2\psi+3k\phi=0,\\ \nonumber
-2\phi_x\psi+\phi_t-2\phi\psi_x-2\phi A+3\phi_{xx}=0,\\ \nonumber
2\psi\psi_x-2\phi B+2A_x-\psi_{xx}+\psi_t=0,\\ \nonumber
-\frac13\phi^2\phi_x+\frac13k\phi^2+k\phi\psi=3k\phi_x+k_x\phi,\\ \nonumber
\frac23\phi^2\psi_x-\frac23\phi\phi_{xx}-\frac83\phi_x^2-2k\psi_x-2kA+2\phi\phi_x\psi\\ \nonumber
=-\frac23\phi\phi_t-2k\phi_x-k_x\phi+k_x\psi,\\ \nonumber
-2\phi_xA+\phi_{xxx}+2\phi\psi\psi_x+kA-\phi_{xx}\psi-4\phi_x\psi_x-\phi\psi_{xx}+2k\psi_x\\ \nonumber
=\phi_{tx}-\phi_t\psi-\phi\psi_t-k_x\psi+3kB,\\ 
A_{xx}-2\psi_xA=A_t+2\phi_xB-2kB,\nonumber \\
-2\psi_xB+B_{xx}=B_t. \label{sysDetEqsCondSymReduced}
\end{gather}
Now from the first equation of~\eqref{sysDetEqsCondSymReduced} it is obvious that either (i) $\phi=0$ or (ii) $\phi_t=0$, $k=\frac29\phi^2$.
Consider separately these two possibilities.

\bigskip\noindent{\it Case (i)}. $\phi(t,x)=0$.
System~\eqref{sysDetEqsCondSymReduced} is read now like
\begin{gather*}
2A_x-\psi_{xx}+2\psi\psi_x+\psi_t=0,\\
-2k\psi_x-2kA-\psi k_x=0,\\
2k\psi_x+kA+k_x\psi-3kB=0,\\
A_{xx}-2A\psi_x+2kB-A_t=0,\\
B_{xx}-2\psi_xB-B_t=0.
\end{gather*}
From the second and third equations we deduce that $A=-3B$. After substituting this to the previous system we obtain
\begin{gather*}
2\psi\psi_x-6B_x-\psi_{xx}+\psi_t=0,\\
k_x\psi+2k\psi_x-6kB=0,\\
6\psi_xB-3B_{xx}+3B_t+2kB=0,\\
B_{xx}-2B\psi_x-B_t=0.
\end{gather*}
It follows from the last two equations that $B=0$. Then the rest of the determining equations is read like
\begin{gather*}
2\psi\psi_x-\psi_{xx}+\psi_t=0,\quad
k_x\psi+2k\psi_x=0.
\end{gather*}
General solution of this system is
\begin{gather*}
k=c,\quad \psi=\const,\\
k=\frac c{(ax+b)^2},\quad \psi=\frac{ax+b}{2ta+m},\quad\mbox{and}\\
k=\frac c{\psi^2},\quad \psi'=\psi^2+a.
\end{gather*}
Nonclassical symmetry operator obtained from the first two branches of the solution of the above system are equivalent to the usual Lie symmetry.
The third branch (up to equivalence transformations of scaling and translations of~$x$) gives cases 1--4 of the theorem.

\bigskip\noindent{\it Case (ii)}. $\phi_t=0$, $k=\frac29\phi^2$.
Substituting this to system~\eqref{sysDetEqsCondSymReduced} we obtain easily that $\psi_t=A_t=B_t=0$.
Then, the rest of the system~\eqref{sysDetEqsCondSymReduced} has the form
\begin{gather*}
-4\phi_x+2\phi\psi+\frac23\phi^2=0,\\
-2\phi_x\psi-2\psi_x\phi-2\phi A+3\phi_{xx}=0,\\
2\psi\psi_x-2\phi B+2A_x-\psi_{xx}=0,\\
\frac{2}9\phi^2\psi_x-\frac23\phi\phi_{xx}-\frac83\phi_{x}^2-\frac49\phi^2A+\frac{14}9\phi\psi\phi_x
=-\frac89\phi^2\phi_x,\\
\phi_{xxx}+\frac49\phi^2\psi_x-\frac23\phi^2B+\frac29\phi^2A-\phi\psi_{xx}-\phi_{xx}\psi-4\phi_x\psi_x
-2\phi_xA+2\phi\psi\psi_x\\=-\frac49\phi\phi_x\psi-\phi\psi_t,\\
A_{xx}-2\psi_xA=-\frac49\phi^2B+2\phi_xB+A_t,\\
B_{xx}=2\psi_xB+B_t.
\end{gather*}
It follows then that
\[
B=\frac{9(2\psi_xA-A_{xx})}{2(2\phi^2-9\phi_x)} ,\quad \psi=\frac{6\phi_x-\phi^2}{3\phi},\quad
A=\frac{4\phi\phi_x-3\phi_{xx}}{6\phi}.
\]
Substituting this values to the above system we obtain a system of 4 differential equations for
one function $\phi$ only that has  first order differential consequence of form
$
3\phi_x^2+\phi^2\phi_x=0.
$
It is not difficult to show that general solution of this constraint $\phi=\frac3{x+c}$, $\phi=c$
satisfies the whole system for $\phi$.
These two values of $\phi$ (taken up to equivalence transformations) give respectively cases 6 and 5 of the theorem.
\end{proof}

\begin{note}
Cases 1, 2 and 4 with $c>0$ were known in~\cite{Bradshaw-Hajek2004,Bradshaw-Hajek&Edwards&Broadbridge&Williams2007},
constant coefficient case 5 with $c=2$ can be found in, e.g.,~\cite{Hydon2000},
while 1,2 and~4 with $c<0$, 3 and 6 are new.
\end{note}

Now we turn into the case 2. That is, we consider nonclassical symmetry operator with $\tau=0$.
Without loss of generality we assume it to be of the form
\[
Q=\p_x+\eta(t,x,u)\p_u.
\]

Any operator nonclassical symmetry of the above form satisfies the following equation
\begin{equation}\label{eqDetEqCondSymTau=0}
-\eta_{xx}-2\eta\eta_{xu}-\eta^2\eta_{uu}+k\eta_uu^2-k\eta_uu^3+\eta_t+k_xu^3-k_xu^2-2k\eta u+3k\eta u^2=0.
\end{equation}
The corresponding invariant surface condition is
$
u_x=\eta .
$
Eliminating $u_x$ and $u_{xx}$, equation \eqref{eqReacDifEqCubicSpace} reads
\begin{equation}\label{newequation}
u_t=\eta\eta_u+\eta_x+k(x)u^2(1-u).
\end{equation}
Using a solution of \eqref{eqDetEqCondSymTau=0}, we can find $u$ by integrating invariance surface condition and then substituting in \eqref{newequation}
to derive a solution of \eqref{eqReacDifEqCubicSpace}.

Now as it was shown in~\cite{Popovych1998,Zhdanov&Lahno1998}  for more general case of $(1+n)$-dimensional evolution equations,
integration of equation \eqref{eqDetEqCondSymTau=0}
is, in some sense, equivalent to integration of the initial
equation~\eqref{eqReacDifEqCubicSpace}.
However, since it contains bigger number of unknown variables it is possible to construct certain partial solutions.
Thus, for example, we have succeeded to find  all ($G^{\sim}$-inequivalent) partial solutions of equation~\eqref{eqDetEqCondSymTau=0} of the form
\[
\eta(x,t,u)=\sum_{p=-m}^n \phi_p(x,t)u^p,
\]
where $m$ and $n$ are positive integers and $\phi_p(x,t)$ unknown functions. We find the following results:

\begin{enumerate}
\item $k=2B^2$, $\eta=B(x)u^2-\tan x u$, where
\[
-4BB'+4B'\tan x-B''+2B+2B^2\tan x=0.
\]
For example, this equation has a solution of form $B=\tan x$.

\item $k=2B^2$, $\eta=B(x)u^2+\tanh x u$, where
\[
4BB'+4B'\tanh x+B''+2B+2B^2\tanh x=0.
\]
In particular, this equation has a solution of form $B=-\tanh x$.

\item $k=2B^2$, $\eta=B(x)u^2+\coth x u$, where
\[
4BB'+4B'\coth x+B''+2B+2B^2\coth x=0.
\]
In particular, this equation has a solution of form $B=-\coth x$.

\item $k=2B^2$, $\eta=B(x)u^2+\frac u x$, where
\[
4xBB'+4B'+xB''+2B^2=0.
\]
A solution of this equation is $B=-\frac1x$.

\item $k=\frac2{x^2}$, $\eta=\frac1x(u^2-1)$.

\item $k=\frac1{2x^2}$, $\eta=\frac1{2x}u^2$.

\item $k=2\tan^2 2x$, $\eta=-u^2\tan 2x$.

\item $k=2\tanh^2 2x$, $\eta=u^2\tanh 2x$.

\end{enumerate}

More detailed investigation of conditional symmetries and construction of associate similarity
solutions of equations from class~\eqref{eqReacDifEqCubicSpace}
will be the subject of a forthcoming paper.

\subsection*{Acknowledgements}
This research was supported by Cyprus Research Promotion Foundation
(project number $\Pi$PO$\Sigma$E$\Lambda$KY$\Sigma$H/$\Pi$PONE/0308/01).

\end{document}